\documentclass{amsart} 
\usepackage[utf8]{inputenc}
\usepackage[T1]{fontenc}
\usepackage[top=2.54cm,bottom=2cm,right=3.5cm,left=3.5cm]{geometry}
\usepackage{color}
\usepackage{indentfirst}
\usepackage{lmodern} \normalfont
\DeclareFontShape{T1}{lmr}{bx}{sc} { <-> ssub * cmr/bx/sc }{}
\usepackage{amsmath}
\usepackage{amsfonts}
\usepackage{derivative}
\usepackage{amssymb}
\usepackage{amsrefs}

\usepackage{lipsum}                     
\usepackage{xargs}                      
\usepackage[pdftex,dvipsnames]{xcolor}

\usepackage{listings}
\usepackage{amsmath}
\usepackage{amsthm,amsfonts,mathrsfs,amsfonts}
\usepackage{slashed}
\usepackage{breqn}
\usepackage{pb-diagram}
\usepackage{bbm}
\usepackage[all]{xy}
\usepackage{times}
\usepackage{amsaddr}
\usepackage{microtype}
\usepackage{enumerate}
\usepackage[super]{nth}
\usepackage{multicol}
\usepackage{tikz-cd}
\usepackage{array}
\usepackage[colorinlistoftodos]{todonotes}

\usepackage{url}
\usepackage{verbatim}

\usepackage{hyperref}
\usepackage[misc]{ifsym}

\newcommand{\ri}{{\rm i}}

\newcommand{\rG}{{\rm G}}

\newcommand{\rI}{{\rm I}}

\newcommand{\rV}{{\rm V}}

\newcommand{\rX}{{\rm X}}

\newcommand{\cW}{\mathcal{W}}

\newcommand{\fg}{{\mathfrak g}}

\newcommand{\fh}{{\mathfrak h}}

\newcommand{\fm}{{\mathfrak m}}

\newcommand{\fu}{{\mathfrak u}}

\newcommand{\Z}{\mathbb{Z}}

\newcommand{\R}{\mathbb{R}}
\newcommand{\C}{\mathbb{C}}

\newcommand{\Oc}{\mathbb{O}}

\newcommand{\su}{\mathfrak{su}}
\newcommand{\so}{\mathfrak{so}}

\newcommand{\gl}{\mathfrak{gl}}
\newcommand{\fsl}{\mathfrak{sl}}

\newcommand{\fsp}{\mathfrak{sp}}

\newcommand{\SO}{{\rm SO}}
\newcommand{\Sp}{{\rm Sp}}

\newcommand{\SU}{{\rm SU}}

\newcommand{\Scal}{{\rm Scal}}

\newcommand{\Ad}{\mathrm{Ad}}

\renewcommand{\det}{\mathop\mathrm{det}\nolimits}

\newcommand{\End}{{\mathrm{End}}}

\renewcommand{\epsilon}{\varepsilon}

\newcommand{\Lie}{\mathrm{Lie}}

\newcommand{\ad}{\mathrm{ad}}

\newcommand{\tr}{\mathop{\mathrm{tr}}\nolimits}

\newcommand{\vol}{\mathrm{vol}}

\newcommand{\qandq}{\quad\text{and}\quad}
\newcommand{\qwithq}{\quad\text{with}\quad}

\def\<{\mathopen{}\left<}
\def\>{\right>\mathclose{}}
\def\({\mathopen{}\left(}
\def\){\right)\mathclose{}}

\usepackage{multicol, color}

\definecolor{gold}{rgb}{0.85,.66,0}
\definecolor{cherry}{rgb}{0.9,.1,.2}
\definecolor{burgundy}{rgb}{0.8,.2,.2}
\definecolor{orangered}{rgb}{0.85,.3,0}
\definecolor{orange}{rgb}{0.85,.4,0}
\definecolor{olive}{rgb}{.45,.4,0}
\definecolor{lime}{rgb}{.6,.9,0}
\definecolor{green}{rgb}{.2,.7,0}
\definecolor{grey}{rgb}{.4,.4,.2}
\definecolor{brown}{rgb}{.4,.3,.1}

\newtheorem{theorem}{Theorem}[section]
\newtheorem*{theorem*}{Theorem}
\newtheorem{corollary}[theorem]{Corollary}
\newtheorem{definition}[theorem]{Definition}
\newtheorem{example}[theorem]{Example}
\newtheorem{lemma}[theorem]{Lemma}
\newtheorem{remark}[theorem]{Remark}

\newtheorem*{remark*}{Remark}
\newtheorem{proposition}[theorem]{Proposition}

\newcommand{\gt}{\mathrm{G}_2}

\DeclareMathOperator{\Gl}{Gl}
\DeclareMathOperator{\real}{Re}

\DeclareMathOperator{\imag}{Im}

\DeclareMathOperator{\Span}{Span}
\DeclareMathOperator{\sym}{sym}

\DeclareMathOperator{\diver}{div}
\DeclareMathOperator{\ricci}{Ric}

\newcommand{\qforq}{\quad \text{for} \quad}

\newcommand{\qwhereq}{\quad \text{where} \quad}

\title{Harmonic $\gt$-structures on almost Abelian Lie groups}
\author{Andrés J. Moreno}
\email{amoreno@unicamp.br}
\address{IMECC - University of Campinas}
\date{\today}

\begin{document}
\begin{abstract}
    We consider left-invariant $\gt$-structures on $7$-dimensional almost Abelian Lie groups. Also, we characterise the associated torsion forms and the full torsion tensor according to the Lie bracket $A$ of the corresponding Lie algebra. In those terms, we establish the algebraic condition on $A$ for each of the possible $16$-torsion classes of a $\gt$-structure. In particular, we show that four of those torsion classes are not admissible, since $\tau_3=0$ implies $\tau_0=0$. Finally, we use the above results to provide the algebraic criteria on $A$, satisfying the harmonic condition $\diver T=0$, and this allows to identify which torsion classes are harmonic.\\

    \noindent 2020 Mathematical Subject Classification.  53C15, 22E25, 53C30.\\

    \noindent \textit{Keywords:} $\gt$-structures, $\SU(3)$-structures,  harmonic structures, Lie groups and Lie algebras.
\end{abstract}

\maketitle
\tableofcontents

\section{Introduction}

The $8$-dimensional, non-associative, normed algebra of octonions $\Oc$ induces an inner $\langle\cdot ,\cdot\rangle$ and cross $\times$ product on $\R^7$, using the decomposition $\Oc=\real \Oc\oplus \imag\Oc\simeq \R\oplus\R^7$, both products $\langle\cdot,\cdot\rangle$ and $\times$ are defined in terms of the octonionic product \cite{harvey1982calibrated}*{Section IV} by
\begin{equation*}
    \langle u,v\rangle =-\real(uv) \qandq u\times v=\imag(uv), \qforq u,v\in \R^7\simeq\imag\Oc .
\end{equation*}
These yield an alternating $3$-form $\varphi_0\in \Lambda^3(\R^7)^\ast$, given by $\varphi_0(u,v,w)=\langle u\times v,w\rangle$ and characterise the Lie group $\gt\subset \SO(7)$ as the stabiliser subgroup of $\varphi_0$ within $\Gl(\R^7)$ (cf. \cite{joyce2000compact}*{Definition 10.1.1}). For an oriented $7$-manifold $M$, a \emph{$\gt$-structure} $\varphi$ is a differential $3$-form on $M$, such that, for each $p\in M$, there exists an oriented isomorphism between $T_pM$ and $\R^7$ identifying $\varphi_p$ and $\varphi_0$.

The space of $\gt$-structures on $M$ is denoted by $\Omega^3_+\subset\Omega^3(M)$, and a relevant geometric implication is that each $\varphi\in \Omega^3_+$ induces a Riemannian metric, $g_\varphi$ and an orientation, $\vol_\varphi$ on $M$, according to the non-linear relation (for details, see \cite{karigiannis2009flows}):
\begin{equation}\label{eq: varphi_metric_volume_relation}
    (u\lrcorner\varphi)\wedge (v\lrcorner\varphi)\wedge\varphi=6g_\varphi(u,v)\vol_\varphi.
\end{equation}

As a consequence, there is a $\varphi$-dependent Hodge star operator $\ast$ and a Riemannian connection $\nabla$ on $M$. The latter leads us to consider the tensor $\nabla\varphi$, which is named the \emph{intrinsic torsion} of $\varphi$. A $\gt$-structure with vanishing intrinsic torsion is called \emph{torsion-free}, and its metric has holonomy contained in the Lie group $\gt$ and, as a result, the metric is Ricci flat. 
 In \cite{fernandez1982riemannian}, Fernández and Gray proved that torsion-free $\gt$-structures are equivalent with the \emph{closed} $d\varphi=0$ and \emph{coclosed} $d\ast\varphi=0$ conditions.     
The construction of non-trivial, torsion-free $\gt$-structures is a challenging task, and the few  known results, including non-compact explicit examples, have arisen from Lie theory, algebraic geometry, and geometric analysis techniques (e.g. \cites{Bryant1989,Chiossi2002,corti1207g2,Joyce1996,JoyceKarigiannis2017}). In the case of homogeneous spaces $M^7=G/K$, the examples with $G$-invariant torsion-free $\gt$-structures are diffeomorphic to the direct product of a flat torus and an Euclidean space \cite{alekseevskii1975}. Thus, in cases where the torsion-free condition is trivial, it is interesting to consider some relaxed torsion conditions. An approach consists of  considering the $16$-torsion classes, via the so-called \emph{ torsion forms}  $\tau_k\in \Omega^k$ ( $k=0,1,2,3$ ) \cite{bryant2003some}
\begin{equation}\label{eq: torsion forms defi}
    d\varphi=\tau_0\psi+3\tau_1\wedge\varphi+\ast\tau_3 \qquad d\psi=4\tau_1\wedge\psi+\tau_2\wedge\varphi,
\end{equation}
where $\psi=\ast\varphi$. The torsion forms are related with the intrinsic torsion $\nabla_i\varphi_{jkl}=T_{ip}{\psi^p}_{jkl}$, via the \emph{total torsion tensor} (see \cite{karigiannis2009flows}):
\begin{equation*}\label{torsion_tensor_introduction}
    T_\varphi=\frac{\tau_0}{4}g-\frac{1}{4}\jmath(\tau_{3})-\ast(\tau_1\wedge\psi)-\frac{1}{2}\tau_2,
\end{equation*} 
where $\jmath: \Omega^3(M)\rightarrow \sym^2(T^\ast M)$ is defined by $j(\tau)(u,v)=\ast(u\lrcorner\varphi\wedge v\lrcorner\varphi\wedge\tau)$. An alternative class of $\gt$-structures emerges as critical points of the energy functional
\begin{equation}\label{eq: energy functional}
    E(\varphi):=\int_M|T_\varphi|^2\vol,
\end{equation}
restricted to the \emph{isometric class} of $\varphi$, i.e., the subset $[[\varphi]]\subset\Omega^3_+$ with the same induced Riemannian metric and orientation as $\varphi$. The critical points of \eqref{eq: energy functional} are characterised by the divergence-free condition of the corresponding total torsion tensor $\diver T_\varphi:=\nabla^iT_{ij}=0$ and, since they are a particular case of the broad theory of harmonic geometric structure \cite{Gonzalez2009}, they are called \emph{harmonic} $\gt$-structures. Note that the torsion-free $\gt$-structures are global minimisers of the energy functional, in this sense, we can think of the \emph{divergence condition} $\diver T=0$ as a natural generalisation of torsion-free $\gt$-structures. In particular, harmonic $\gt$-structures are stationary points of the corresponding gradient flow 
\begin{equation}\label{isometric_flow}
	    \frac{\partial \varphi_t}{\partial t}=\diver T_t\lrcorner \psi_t, \quad g_{\varphi_t}=g_0 \quad \forall \quad t.
	\end{equation} 
The harmonic condition is transversal to the $16$ torsion classes of $\gt$-structures. However, S. Grigorian proved that some torsion classes have divergence-free torsion tensor \cite{Grigorian2019}*{Theorem 4.3}, namely $\diver T=0$ if one of the following holds:
\begin{align}\label{eq: Grigorian_divergence_free_cases}
\begin{split}
   (\ri)& \text{ $\tau_0$ constant, $\tau_1=0$ and arbitrary $\tau_2$ and $\tau_3$.}\\
   (\ri\ri)& \text{ $\tau_1$ arbitrary and $\tau_0$, $\tau_2$ and $\tau_3$ vanishing.}
\end{split}
\end{align}
For the flow \eqref{isometric_flow},  some analytical properties have been studied, such as short-time existence, uniqueness and compactness of solutions, as well as long-time existence as long as the torsion remains bounded along the flow \cites{ Grigorian2019, Karigiannis2019, loubeau2019}. For the particular case of the $7$-sphere, explicit examples of harmonic $\Sp(2)$-invariant $\gt$-structures were obtained in \cite{loubeau2021}, with non-zero torsion forms $\tau_0$, $\tau_1$, and $\tau_3$, as well as explicit flow solutions of  \eqref{isometric_flow}. And for  a class of solvable Lie groups, in \cite{garrone2022} the author finds new examples of harmonic invariant $\gt$-structures with $\tau_k\neq 0$ for $k=0,1,2,3$. 

In this work, we study harmonic $\gt$-structures on almost Abelian Lie groups, according to their torsion classes. The $n$-dimensional Lie group $G$ is called \emph{almost Abelian} if it admits an Abelian normal subgroup of codimension $1$. Analogously, the corresponding Lie algebra $\fg:=\Lie(G)$ is almost Abelian if it has a codimension $1$ Abelian ideal $\fh$. Its Lie bracket is then completely encoded by $A\in \gl(\fh)$, in other words
\begin{align}\label{eq: almost_abelian_Lie_bracket}
    [e_n, u]=Au \qandq [u,v]=0, \qforq \text{any} \quad u,v\in\fh,
\end{align}
and $e_n\in \fg\backslash \fh$.  When $\dim \fg=7$, the Lie group $G$ admits a $\gt$-structure, induced by the left (or right) Lie group translation of a $\gt$-structure on $\fg$, given by:
\begin{equation}\label{coclosed_almost_abelian_G2}
\varphi=\omega\wedge e^7+\rho^+=e^{127}+e^{347}+e^{567}+e^{135}-e^{146}-e^{245}-e^{236},
\end{equation}
where
$$
\omega=e^{12}+e^{34}+e^{56} \qandq \rho_+=e^{135}-e^{146}-e^{245}-e^{236}
$$
are the canonical $\SU(3)$--structure of $\mathbb{C}^3\simeq \mathfrak{h}$ and $e^{ijk}$ denotes $e^i\wedge e^j\wedge e^k$. Additionally, we have the natural dual $4$-form on $\fg$
\begin{equation}
\psi:=\ast\varphi=\frac{1}{2}\omega^2+\rho_-\wedge e^7=e^{1234}+e^{1256}+e^{3456}-e^{2467}+e^{2357}+e^{1457}+e^{1367},
\end{equation}
where $\rho_-=J^\ast\rho_+$ and $J$ is the canonical almost complex structure on $\mathbb{R}^6$, defined by $\omega:=\langle J\cdot,\cdot\rangle$. It was proved in \cite{Freibert2013}*{Proposition 2.8 (a)} that any $\gt$-structure $\widetilde{\varphi}$ on an almost Abelian Lie algebra $(\fg,A)$ can be written in the form \eqref{coclosed_almost_abelian_G2}, for an \emph{adapted basis} $f_1,\dots,f_7$ of $\fg$ such that $f_7\in\fg\backslash\fh$ and $\fh=\Span(f_1,...,f_6)$, since $\Gl(\fg)_{\widetilde{\varphi}}\simeq \gt$ acts transitively on  one and six-dimensional subspaces of $\fg$. Almost Abelian Lie groups have been studied for a wide variety of geometric structures, namely, Hermitian, symplectic, and Kähler structures, as well as for geometrical problems such as the Ricci flow and the Laplacian (Co-) flow of $\gt$-structures \cites{Arroyo2013,Fino2018,fino2020,Lauret2017,andrada2018}. In particular,   M. Freibert proved that $(\fg,A,\varphi)$ is an almost Abelian Lie algebra with a closed or coclosed $\gt$-structure, if and only if, the bracket $A$ belongs to $\fsl(3,\C)$ or $\fsp(6,\R) $, respectively \cites{Freibert2012,Freibert2013}. 

In Proposition \ref{torsion_forms}, we describe the torsion forms of an almost Abelian Lie algebra with $\gt$-structure $(\fg,\varphi)$ in terms of its Lie bracket $A$. We obtain
	\begin{align}\label{eq: torsion_forms_introduction}
	\begin{split}
	\tau_0&=\frac{2}{7}\tr(JA) \qquad  \tau_2= -\frac{1}{3}\left(\begin{array}{c|c}
	2[J,C_A]-\tr(A)J+3(JA^t+AJ) & -J\alpha^\sharp \\ \hline
	\Big(J\alpha^\sharp\Big)^t & 0
	\end{array}
	\right)\\
	\tau_1&=\frac{1}{12}\alpha-\frac{1}{6}\tr(A)e^7 \qquad \frac{1}{4}\jmath(\tau_3)= \left(\begin{array}{c|c}
	\frac{1}{14}\tr(JA)I_6-\frac{1}{2}[J,S_A] & \frac{1}{4}J\alpha^\sharp \\ \hline
	\frac{1}{4}\Big(J\alpha^\sharp\Big)^t & -\frac{3}{7}\tr(JA)
	\end{array}
	\right),
	\end{split}
	\end{align}
	where $\alpha\in \fh^\ast$ is defined in equation \eqref{alpha_form} and the blocks correspond to the $\SU(3)$-decomposition of $\gl(\R^6)$ \cite{bedulli2007}*{Section 2.1 and 2.3}:
\begin{equation}\label{eq: splitting.gl(6)}
    \gl(\R^6)=\R\cdot I_6\oplus \sym^0_+(\R^6)\oplus \sym^0_-(\R^6)\oplus\R\cdot J\oplus \su(3)\oplus \fm 
\end{equation} 
    with
\begin{align*}
    \sym_-^0=&\{A\in \gl(\R^6); \quad A^t=A \qandq \{J,A\}=0 \} \\
    \sym_+^0=&\{A\in \gl(\R^6); \quad A^t=A,\quad \tr(A)=0 \qandq [J,A]=0\} \\
    \su(3)=&\{A\in \gl(\R^6); \quad A^t=-A,\quad \tr(JA)=0 \qandq [J,A]=0\} \\
    \fm=&\{A\in \gl(\R^6); \quad A^t=-A, \qandq \{J,A\}=0\} .
\end{align*}
In consequence, we obtain the full torsion tensor with respect to the Lie bracket $A\in \gl(\R^6)$
\begin{equation}\label{eq: torsion_tensor_A_intro}
	T= \frac{1}{2}\left(\begin{array}{c|c}
	[J,S(A)]+[J,C(A)]+(JA^t+AJ) & -J\alpha(A)^\sharp \\ \hline
	0 & \tr(JA)
	\end{array}
	\right).
	\end{equation}

As a first application of \eqref{eq: torsion_forms_introduction} in Theorem \ref{thm: 16-torsion classes}, we describe the possible torsion classes of an invariant $\gt$-structure on almost Abelian Lie algebras:

\begin{table}[h!]
\begin{tabular}{||c|c|c||}
\hline 
\text{Class} &  \text{Vanishing torsion} \quad & \text{Bracket relation}    \\ [0.5ex] \hline\hline
    $\cW=\{0\}$ & $\tau_0=0,\tau_1=0,\tau_2=0,\tau_3=0$  &  $A\in\su(3)$ \\ [0.5ex]
    $\cW_4$ & $\tau_0=0,\tau_2=0,\tau_3=0$ &  $A\in\R\cdot I_6\oplus\su(3)$\\ [0.5ex]
    $\cW_2$ & $\tau_0=0,\tau_1=0,\tau_3=0$ & $A\in \sym_+^0\oplus \su(3)$\\ [0.5ex]
    $\cW_3$ & $\tau_0=0,\tau_1=0,\tau_2=0$ & $A\in \sym_-^0\oplus\su(3)$\\ [0.5ex]
    $\cW_1\oplus\cW_3$ & $\tau_1=0,\tau_2=0$ & $A\in \sym_-^0\oplus \R\cdot J\oplus \su(3)$\\ [0.5ex]
    $\cW_2\oplus\cW_4$ & $\tau_0=0,\tau_3=0$ & $A\in\R\cdot I_6\oplus \sym_+^0\oplus\su(3)$ \\ [0.5ex]
    $\cW_3\oplus\cW_4$ & $\tau_0=0,\tau_2=0$ & $A\in\R\cdot I_6\oplus \sym_-^0\oplus \su(3)$ \\ [0.5ex]
    $\cW_2\oplus\cW_3$ & $\tau_0=0,\tau_1=0$ & $A\in \sym_+^0\oplus \sym_-^0\oplus \su(3)$ \\ [0.5ex]
    $\cW_1\oplus\cW_3\oplus\cW_4$ & $\tau_2=0$ & $A\in\R\cdot I_6\oplus \sym_-^0\oplus \R\cdot J\oplus\su(3)$ \\ [0.5ex]
    $\cW_1\oplus\cW_2\oplus\cW_3$ & $\tau_1=0$ &  $A\in \sym_+^0\oplus \sym_-^0\oplus \R\cdot J\oplus\su(3)$ \\ [0.5ex]
    $\cW_2\oplus\cW_3\oplus\cW_4$ & $\tau_0=0$ & $A\in\sym(\R^6)\oplus\fm\oplus \su(3)$ \\ [0.5ex]
    $\cW_1\oplus\cW_2\oplus\cW_3\oplus\cW_4$ & \text{No vanishing condition} & $A\in \gl(\R^6)$ \\ [1ex] \hline
\end{tabular}
\caption{Torsion classes}\label{tb:torsion_classes_introduction}
\end{table}

Finally, using \eqref{eq: torsion_tensor_A_intro}, we express the  divergence-free condition $\diver T=0$ as an algebraic condition on $A\in\gl(\R^6)$ (see Theorem \ref{divergence_vector} or equivalently, \eqref{eq: vector_divergence condition})
\begin{equation*}
        \tr(JA)\tr(A)=0 \qandq -\tr(A)J^\ast\alpha(A)+\theta(C(A))(J^\ast\alpha(A))=0.
    \end{equation*}

Hence, following the notation of Table \ref{tb:torsion_classes_introduction}, we prove in Theorem \ref{thm:harmonic_torsion_classes}, 
that $\varphi$ is harmonic if its torsion belongs to one of the following torsion classes:
\begin{eqnarray*}
     \{0\}, \quad \cW_2, \quad \cW_3, \quad \cW_4,\\ \cW_1\oplus\cW_3, \quad \cW_2\oplus\cW_4, \quad \cW_3\oplus\cW_4, \\ \cW_2\oplus\cW_3, \quad \cW_1\oplus\cW_2\oplus\cW_3. 
\end{eqnarray*}
 Moreover, if $\varphi$ is of type $\cW_1\oplus\cW_3\oplus\cW_4$ and $\diver T=0$, then $\varphi$ is of type $\cW_1\oplus\cW_3$ or $\cW_3\oplus\cW_4$.

\subsection*{Convention:}  Let $(M,g)$ be  a $n$-dimensional Riemannian manifold. 
 A differential $k$-form $\alpha$ on $M$ will be written as
$$
\alpha=\frac{1}{k !} \alpha_{i_1 i_2 \cdots i_k}  dx^{i_1} \wedge  \cdots \wedge dx^{i_k}
$$
in local coordinates $\left(x^1, \ldots, x^n\right)$, where $\alpha_{i_1 i_2 \cdots i_k}$ is completely skew-symmetric in its indices. According to this, for a given coordinate vector field $\partial_m$ and a $k$-form $\alpha$, their interior product is the $(k-1)$-form
$$
\left.\partial_m\right\lrcorner \alpha=\frac{1}{(k-1) !} \alpha_{m i_1 i_2 \cdots i_{k-1}} dx^{i_1} \wedge \cdots \wedge d x^{i_{k-1}}.
$$
Given the endomorphisms $A,B\in \End(TM)$, the corresponding $2$-tensor has coefficients $A_{ij}=A_i^lg_{lj}$, where $g_{lj}$ denotes the coefficients of the metric. Conversely, we have $A_i^l=A_{ij}g^{jl}$ where $g^{jl}$ denotes the coefficients of inverse of the metric $g$. Thus, the product has coefficients 
$$
(AB)_i^k=A^k_jB_i^j \qandq (AB)_{ik}=B_{ij}g^{jl}A_{lk}.
$$

\bigskip

\noindent\textbf{Acknowledgements:}
The author is grateful to the anonymous referees for their valuable review. Also, thanks to Viviana del Barco, Jorge Lauret and Alejandro Tocalchier for their very helpful  comments. This work has been funded by the São Paulo Research Foundation (Fapesp)  \mbox{[2021/08026-5]}.

\section{$\SU(3)$-invariant decomposition}

In this section, we describe the $\SU(3)$-invariant components of $\gl(\R^6)$, in order to derive the formulae \eqref{eq: torsion_forms_introduction} with respect to each $\SU(3)$-irreducible component of the Lie bracket $A\in \gl(\R^6)$, of an almost Abelian Lie algebra. Intending to achieve it, we follow some results from \cites{bedulli2007,fino2017closed} related to $\SU(3)$-structures.

Let $(\omega,\rho^+,h)$ be a $\SU(3)$-structure on $\R^6$, in which $\omega\in \Lambda^2(\R^6)^\ast$ is the fundamental $2$-form compatible with the inner product $h$, i.e., $\omega(\cdot,\cdot)=h(J\cdot,\cdot)$ with $J$ the almost complex structure of $\C^3\simeq \R^6$ and $\rho^+\in \Lambda^3(\R^6)^\ast$ is the real part of the complex volume form $P:=\rho^++iJ^\ast\rho^+=\rho^++i\rho^-$
 on $\C^3$. Moreover, the forms $\omega, \rho^+$ and $P$ are interconnected by the relation:
 \begin{equation}\label{eq: volume6}
     \vol_6=\frac{\omega^3}{6}=\frac{1}{4}\rho^+\wedge\rho^-=\frac{i}{8}P\wedge \bar{P} \qandq \omega\wedge\rho^+=\omega\wedge\rho^-=0,
 \end{equation}
 and the Hodge star operator $\star: \Lambda^k(\R^6)^\ast\rightarrow\Lambda^{6-k}(\R^6)^\ast$ induced by $h$ and $\vol_6$ satisfies:
 \begin{equation}\label{eq: Hodge_star_6}
     \star^2\sigma=(-1)^k\sigma \qandq  \star(\beta\wedge\sigma)=(-1)^k\beta^\sharp\lrcorner\star\sigma, \qforq \sigma\in \Lambda^k(\R^6)^\ast, \quad \beta\in (\R^6)^\ast,
 \end{equation}
 where $\sharp:(\R^6)^\ast\rightarrow \R^6$ is the musical isomorphism induced by $h$, defined by $\beta(u)=h(\beta^\sharp,u)$. As a consequence of the relations \eqref{eq: volume6} and \eqref{eq: Hodge_star_6}, there are the following properties:
 
 \begin{lemma}\cite{fino2017closed}\label{Fino_Lemma}
	Let $(\omega,\rho^+,h)$ be a $\SU(3)$-structure on $\R^6$. For a $1$-form $\beta$, we have the following identities: 
	\begin{enumerate}
		\item[(i)] $\star(\beta\wedge\rho_-)\wedge\omega=J^\ast\beta\wedge\rho_+$ and $\star(\beta\wedge\rho_+)\wedge\omega=-J^\ast\beta\wedge\rho_-$.
		\item[(ii)] $\star(\beta\wedge\rho_-)\wedge\omega^2=\star(\beta\wedge\rho_+)\wedge\omega^2=0$.
		\item[(iii)] $\star(\beta\wedge\rho_-)\wedge\rho_+=-\star(\beta\wedge\rho_+)\wedge\rho_-=\beta\wedge\omega^2=2\star J^\ast\beta$.
		\item[(iv)] $\star(\beta\wedge\rho_-)\wedge\rho_-=\star(\beta\wedge\rho_+)\wedge\rho_+=\star(\beta\wedge\omega)\wedge\omega=-J^\ast\beta\wedge\omega^2=2\star\beta$.
		\item[(v)] $(J^\ast\beta)^\sharp=-J(\beta^\sharp)$ and $(\beta^\sharp\lrcorner\omega)^\sharp=J(\beta^\sharp)$.
	\end{enumerate}	
\end{lemma}
Most of the later computations are addressed in coordinates, hence it will be convenient to express $\omega,\rho^+$ and $\rho^-$ as:
 \begin{equation}\label{eq: SU3_structure_coordinates}
     \omega=\frac{1}{2}\omega_{ij}e^{ij}=\frac{1}{2}J^k_{i}h_{kj}e^{ij},\quad \rho^+=\frac{1}{6}\rho^+_{ijk}e^{ijk} \qandq \rho^-=\frac{1}{6}\rho^-_{ijk}e^{ijk},
 \end{equation}
 where $\rho_{ijk}=\rho(e_i,e_j,e_k)$ for $e_1,...,e_6$ a basis of $\R^6$. Thus, according to \eqref{eq: SU3_structure_coordinates}, there hold the following contractions: 
 \begin{lemma}\cite{bedulli2007}*{Section 2.2}\label{SU3_identities}
 Let $e_1,...,e_6$ be a basis of $\R^6$ and $(\omega,\rho^+,h)$ a $\SU(3)$-structure as given in \eqref{eq: SU3_structure_coordinates}. Then, the following identities hold: 
   \begin{align*}
\begin{split}
&\omega_{ip}h^{pq}\omega_{qj}=-h_{ij}\\
&\rho^+_{ijk}\rho^+_{abc}h^{kc}=-\omega_{ia}\omega_{jb}+\omega_{ib}\omega_{ja}+h_{ia}h_{jb}-h_{ja}h_{ib}=	\rho^-_{ijk}\rho^-_{abc}h^{kc}\\
&\rho^+_{ijk}\rho^+_{abc}h^{jb}h^{kc}=4h_{ia}=\rho^-_{ijk}\rho^-_{abc}h^{jb}h^{kc}\\
&\rho^-_{ijk}\rho^+_{abc}h^{kc}=-\omega_{ia}h_{jb}+\omega_{ja}h_{ib}+\omega_{ib}h_{ja}-\omega_{jb}h_{ia}\\
&\rho^+_{ijk}\rho^-_{abc}h^{jb}h^{kc}=4\omega_{ia}\\
&\rho^+_{ijp}h^{pq}\omega_{qk}=\rho^-_{ijk}, \quad \rho^-_{ijp}h^{pq}\omega_{qk}=-\rho^+_{ijk},\quad \rho^+_{ijk}\omega_{bc}h^{jb}h^{kc}=0.
\end{split}
\end{align*}  
 \end{lemma}
With respect to the $\SU(3)$-decomposition \eqref{eq: splitting.gl(6)}, we can write each $A\in \gl(\R^6)$ as:
\begin{equation*}
    A=\frac{\tr(A)}{6}I_6+S_+(A)+S_-(A)+\frac{\tr(JA)}{6}J+C_+(A)+C_-(A),
\end{equation*}
where,
\begin{align*}
     S_+(A)=&-\frac{1}{2}J\{S(A),J\}-\frac{\tr(A)}{6}I_6\in \sym^0_+(\R^6)\\
      S_-(A)=&\quad \frac{1}{2}J[S(A),J]\in \sym^0_-(\R^6) \qwithq S(A)=\frac{1}{2}(A+A^t)\\
      C_+(A)=&-\frac{1}{2}J\{C(A),J\}-\frac{\tr(JA)}{6}J\in \su(3)\\
     C_-(A)=&\quad \frac{1}{2}J[C(A),J]\in \fm \qwithq C(A)=\frac{1}{2}(A-A^t),
\end{align*}
here $[J,A]=JA-AJ$ and $\{J,A\}=JA+AJ$ denote the commutator and the anti-commutator, respectively. Using the Hermitian inner product $h$ on $\R^6$ (i.e. $h(J\cdot,J\cdot)=h(\cdot,\cdot)$), we can identify $\gl(\R^6)$ with bilinear forms on $\R^6$ and consequently, the splitting \eqref{eq: splitting.gl(6)} becomes:
\begin{equation}\label{eq: splitting.bilinear.forms}
    \R^6\otimes (\R^6)^\ast\simeq \R\cdot h\oplus S^2_+\oplus S^2_-\oplus \Lambda^2_1\oplus\Lambda^2_6\oplus \Lambda^2_8
\end{equation}
in which, 
\begin{align*}
    S^2_+=&\{\alpha \in S^2(\R^6)^\ast; \quad J^\ast\alpha=\alpha \qandq \tr_h(\alpha)=0 \}\\
    S^2_-=&\{\alpha \in S^2(\R^6)^\ast; \quad J^\ast\alpha=-\alpha \}\\ \Lambda^2_1=&\R\cdot \omega\\
    \Lambda^2_8=&\{\alpha \in \Lambda^2(\R^6)^\ast; \quad J^\ast\alpha=\alpha \qandq \alpha\wedge\omega^2=0 \}\\
    =&\{\alpha\in \Lambda^2(\R^6)^\ast; \quad \alpha\wedge\rho_+=0 \qandq J^\ast\alpha=-\star(\alpha\wedge\omega)\}\\
    \Lambda^2_6=&\{\alpha \in \Lambda^2(\R^6)^\ast; \quad J^\ast\alpha=-\alpha \}\\
    =&\{u\lrcorner\rho_++Ju\lrcorner\rho_-; \quad u\in \R^6\}.
\end{align*}
Where $\star:\Lambda^k(\R^6)^\ast\rightarrow \Lambda^{6-k}(\R^6)^\ast$ is the Hodge star operator. From \eqref{eq: splitting.gl(6)} and \eqref{eq: splitting.bilinear.forms}, we have the identification via the inner product $h$
\begin{align*}
    \sym^0_\pm(\R^6)\simeq S_\pm^2, \quad \su(3)\simeq\Lambda^2_8 \qandq \fm\simeq \Lambda^2_6.
 \end{align*}
 In general, consider the infinitesimal $\gl(\R^n)$-representation $\theta: \gl(\R^n)\rightarrow \End(\Lambda^k(\R^n)^\ast)$ defined by:
 \begin{align}\label{eq: theta_representation}
 \begin{split}
     \theta(B)\gamma:=&\frac{d}{dt}|_{t=0}(e^{-Bt})^\ast\gamma=-\gamma(B\cdot,...,\cdot)-....-\gamma(\cdot,....,B\cdot)\\
     =&-\frac{1}{(k-1)!}B^l_{i_1}\gamma_{li_2\cdots i_k}e^{i_1\cdots i_k} \qwhereq \gamma\in\Lambda^k(\R^n)^\ast.
\end{split}
 \end{align}
Regarding \eqref{eq: splitting.gl(6)} and \eqref{eq: theta_representation}, we have the following properties:

\begin{lemma}\label{lem: theta_omega}
Let $(\omega,\rho^+,h)$ be a $\SU(3)$-structure on $\R^6$ and consider the induced maps 
$$
\theta(\quad)\omega: \gl(\R^6)\rightarrow \Lambda^2(\R^6)^\ast \quad \theta(\quad)\rho^+: \gl(\R^6)\rightarrow \Lambda^3(\R^6)^\ast,
$$
then we have that:
\begin{align*}
\theta(\R\cdot I_6)\omega&=\R\cdot \omega,\\  \theta(\R\cdot J)\omega&=\theta(\sym^0_-)\omega= \theta(\su(3))\omega=\{0\},\\ \theta(\sym^0_+)\omega&\subseteq \Lambda^2_8 \qandq \theta(\fm)\omega\subseteq \Lambda^2_6,
\end{align*}
and
\begin{equation*}
    \theta(A)\rho^+=\theta(JA)\rho^-.
\end{equation*}
\end{lemma}

\begin{proof}
 The relations $\theta(\R\cdot I_6)\omega=\R\cdot\omega$ and $\theta(R\cdot J)\omega=\{0\}$ follow from the definition of \eqref{eq: theta_representation}. Also, $\theta(\su(3))\omega=\{0\}$ holds, since $\su(3)\subset\fsl(3,\C)\simeq \mathfrak{stab}(\omega)$, i.e., the Lie algebra stabiliser of $\omega$. Now, if $B\in \sym_-^0$, we have:
 \begin{align*}
     \theta(B)\omega (u,v)&=-\omega(Bu,v)-\omega(u,Bv)\\
     &=-h(JBu,v)-h(BJu,v)\\
     &=-h(\{J,B\}u,v)=0.
 \end{align*}
 Now, for $B\in \sym_+^0$, it is easy to check that $J^\ast(\theta(B)\omega)=\theta(B)\omega$, and also 
\begin{align*}
    \theta(B)\omega\wedge \omega^2=\frac{1}{3}\theta(B)\omega^3=2\theta(B)\vol_6=2\tr(B)\vol_6=0,
\end{align*}
consequently, $\theta(B)\omega\in \Lambda^2_8$. 
By a similar computation, we have $J^\ast\theta(B)\omega=-\theta(B)\omega$ for $B\in \fm$. Finally, using the definition \eqref{eq: theta_representation} and Lemma \ref{SU3_identities}, we have
\begin{align*}
    \theta(A)\rho^+=&-\frac{1}{2}A^l_{i}\rho^+_{ljk}e^{ijk}=\frac{1}{2}A^l_{i}\rho^-_{jkp}h^{pq}\omega_{ql}e^{ijk}\\
    =&-\frac{1}{2}J^p_lA^l_i\rho^-_{pjk}e^{ijk}=\theta(JA)\rho^-
\end{align*}
\end{proof}

Consider the map $\alpha:\gl(\R^6)\rightarrow \Lambda^1(\R^6)^\ast$ defined by 
\begin{align}\label{alpha_form}
	\alpha(A):=&\star(\omega\wedge\theta(A^t)\rho_-)=-\star(\theta(A^t)\omega\wedge\rho_-).
\end{align}
From Lemma \ref{lem: theta_omega}, we see that $\ker\alpha=\sym(\R^6)\oplus \fu(3)$. Hence, the $1$-form \eqref{alpha_form} depends uniquely from the $\fm$-part of $A$, we can see this explicitly as follows: consider the $k$-form
 $\gamma\in\Lambda^k(\R^6)^\ast$, given by $\gamma=\frac{1}{k!}\gamma_{i_1\cdots i_k}e^{i_1\cdots i_k}$, in which $\gamma_{i_1\cdots i_k}=\gamma(e_{i_1},...,e_{i_k})$, and define  the contraction of $\gamma$ by $A\in \gl(\R^6)$ as
\begin{equation*}
    A\lrcorner\gamma=\frac{1}{(k-2)!}A_{ij}h^{ip}h^{jq}\gamma_{pqi_1\cdots i_{k-2}}e^{i_1\cdots i_{k-2}}, \qforq k\geq 2,
\end{equation*}
notice that, if $A$ is symmetric, then  $A\lrcorner\gamma$ vanishes.

\begin{lemma}\label{alpha_representation}
    The $1$-form \eqref{alpha_form} satisfies:
    \begin{equation*}
        \alpha(A)=\frac{1}{2}[J,C(A)]\lrcorner\rho_+=JC_-(A)\lrcorner\rho_+.
    \end{equation*}
    Also, for the dual vector, $\alpha(A)^\sharp$ holds
    \begin{equation*}
        \alpha(A)^\sharp\lrcorner\rho_+(u,v)=2h([J,C(A)]u,v) \qandq \alpha(A)^\sharp\lrcorner\rho_-(u,v)=2h(J[J,C(A)]u,v)
    \end{equation*}
    for any $u,v\in\R^6$.
\end{lemma}
\begin{proof}
    Using the decomposition \eqref{eq: splitting.gl(6)} and Lemma \ref{lem: theta_omega} to the matrix $A^t$, we have $\alpha(A)=\star(\theta(C_-(A))\omega\wedge\rho_-)$. Note that $\theta(C_-(A))\omega(u,v)=-2h(JC_-(A)u,v)=h([C(A),J]u,v)$, then
    \begin{align*}
        \alpha(A) =&-\frac{1}{2}[J,C(A)]_{ij}\star(e^i\wedge e^j\wedge\rho_-)\\
        =&-\frac{1}{2}[J,C(A)]_{ij}h^{ip}h^{jq}(e_p\lrcorner( e_q\lrcorner\rho_+))
        =-\frac{1}{2}[J,C(A)]_{ij}h^{ip}h^{jq}\rho^+_{qpm}e^m\\
        =&\quad \frac{1}{2}[J,C(A)]\lrcorner\rho_+.
    \end{align*}
    For the second part, using the identities \eqref{SU3_identities}, we have
    \begin{align*}
        \alpha(A)^\sharp\lrcorner\rho_+=&\frac{1}{2}\alpha(A)_ih^{il}\rho_{ljk}^+e^{jk}
        =\frac{1}{4}[J,C(A)]_{mn}h^{mp}h^{nq}\rho_{pqi}^+h^{il}\rho_{ljk}^+e^{jk}\\
        =&\frac{1}{4}[J,C(A)]_{mn}h^{mp}h^{nq}(-\omega_{pj}\omega_{qk}+\omega_{pk}\omega_{qj}+h_{pj}h_{qk}-h_{pk}h_{qj})e^{jk}\\
        =&[J,C(A)]_{jk}e^{jk}.
    \end{align*}
     Similarly, we obtain $\alpha(A)^\sharp\lrcorner\rho_-=(J[J,C(A)])_{jk}e^{jk}$.
\end{proof}

We summarise the above discussion with the following corollary.

\begin{corollary}\label{cor: vanishing alpha}
    Let $\alpha(A)$ be the $1$-form given in \eqref{alpha_form}, then $\alpha(A)= 0$ if, and only if, $A\in  \sym(\R^6)\oplus \fu(3)$. That means, if the skew-symmetric part  of $A$ anti commutes with $J$.
\end{corollary}

\section{Almost Abelian Lie groups: the torsion of a  $\gt$-structure}\label{almost_abelian_section}

Let $(\fg,A,\varphi)$ be a $7$-dimensional almost Abelian Lie algebra  with Lie bracket \eqref{eq: almost_abelian_Lie_bracket} and $\gt$-structure \eqref{coclosed_almost_abelian_G2}. 
Denote by $\fh=\R^6$ the codimension $1$ Abelian ideal of $\fg$, the Hodge star operator $\ast:\Lambda^k\mathfrak{g}^\ast\rightarrow\Lambda^{7-k}\mathfrak{g}^\ast$, induced by  $g_\varphi(\cdot,\cdot)=(e^1)^2+\cdots+(e^7)^2$. The following Lemma encodes some geometric relations between $\fg$ and $\R^6$:

\begin{lemma}\cites{Lauret2017,lauret2019}\label{Lauret_Lemmata}
	Denote by $d$ the exterior derivative of  $k$-forms on the $\fg$ with Lie bracket $A\in\gl(\R^6)$, so  for $\gamma\in \Lambda^k\mathfrak{h}^\ast$, the following properties hold:
	\begin{enumerate}
		\item[(i)]  $\ast\gamma=\star\gamma\wedge e^7$ and $\ast(\gamma\wedge e^7)=(-1)^k\star\gamma$.
		\item[(ii)] $d e^7=0$ and $d\gamma=(-1)^k\theta(A)\gamma\wedge e^7$.
		\item[(iii)]  $\theta(A)\star\gamma+\star\theta(A^t)\gamma=-\tr(A)\star\gamma$.
	\end{enumerate}
\end{lemma}
	
According to the relation of the torsion forms with $d\varphi$ and $d\psi$ in \eqref{eq: torsion forms defi}, we have explicitly \cite{Karigiannis2012}*{Eq: (2.2) and (2.3)}:

\begin{equation}\label{eq: general_formulae_torsion}
\tau_0=\frac{1}{7}\ast(\varphi\wedge d\varphi), \quad \tau_1=\frac{1}{12}\ast(\varphi\wedge\ast d\varphi), \quad \tau_2=4\ast(\tau_1\wedge\psi)-\ast d\psi \qandq \tau_3=\ast d\varphi-\tau_0\varphi-3\ast(\tau_1\wedge\varphi).
\end{equation}
Now, we proceed to compute the torsion forms \eqref{eq: general_formulae_torsion} for $(\fg,\varphi)$ in terms of the bracket $A$. 

\begin{proposition}\label{torsion_forms}
	The torsion forms of $(\mathfrak{g},A,\varphi)$ are
	\begin{align*}
	\tau_0=&\frac{2}{7}\tr(JA), \quad \tau_1=-\frac{1}{12}\alpha(A)-\frac{1}{6}\tr(A)e^7,\\
	\tau_2=& -\frac{1}{3}\left(\begin{array}{c|c}
	2[J,C(A)]-\tr(A)J+3(JA^t+AJ) & -J\alpha^\sharp \\ \hline
	\Big(J\alpha^\sharp\Big)^t & 0
	\end{array}
	\right),\\
	\frac{1}{4}\jmath(\tau_3)=& \left(\begin{array}{c|c}
	\frac{1}{14}\tr(JA)I_6-\frac{1}{2}[J,S(A)] & \frac{1}{4}J\alpha^\sharp \\ \hline
	\frac{1}{4}\Big(J\alpha^\sharp\Big)^t & -\frac{3}{7}\tr(JA)
	\end{array}
	\right),
    \end{align*}
	where $\alpha(A)$ is defined in  \eqref{alpha_form}.
\end{proposition}

\begin{proof}
	For the torsion $0$-form, using Lemmata \ref{lem: theta_omega} and \ref{Lauret_Lemmata}, we have: 
	\begin{align*}
		\tau_0=&\frac{1}{7}\ast(\varphi\wedge d\varphi)=-\frac{1}{7}\ast(\rho_+\wedge\theta(A)\rho_+\wedge e^7)\\
		      =&-\frac{1}{7}\star(\theta(JA)\rho_-\wedge\star\rho_-)=-\frac{1}{7}\langle \theta(JA)\rho_-,\rho_-\rangle=\frac{2}{7}\tr(JA).
	\end{align*}
For the torsion $1$-form, we use Lemma \ref{Lauret_Lemmata} and \eqref{alpha_form}
	\begin{align*}
		\tau_1=&\frac{1}{12}\ast(\varphi\wedge\ast d\varphi)=\frac{1}{12}\ast(\varphi\wedge\star\theta(A)\rho_+)\\
		      =&\frac{1}{12}\Big(\star(\omega\wedge \star\theta(A)\rho_+)+\star(\rho_+\wedge\star\theta(A)\rho_+) e^7\Big)\\
		      =&\frac{1}{12}\Big(-\star(\omega\wedge\theta(A^t)\rho_-)+\langle\rho_+,\theta(A)\rho_+\rangle e^7\Big)\\
		      =&-\frac{1}{12}\alpha(A)-\frac{1}{6}\tr(A)e^7.
    \end{align*}
 Replacing the expression of $\tau_1$ into $d\psi$ and using Lemma \ref{Fino_Lemma}, we get the torsion $2$-form 
    \begin{align*}
    	\tau_2
    	      =&-\frac{1}{3}\star(\alpha(A)\wedge\frac{\omega^2}{2})\wedge e^7-\frac{1}{3}\star(\alpha(A)\wedge\rho_-)-\frac{2}{3}\tr(A)\star\Big(\frac{\omega^2}{2}\Big)-\star\theta(A)\frac{\omega^2}{2}\\
    	      =&\frac{1}{3}J^\ast\alpha(A)\wedge e^7-\frac{1}{3}\alpha(A)^\sharp\lrcorner\rho_++\frac{1}{3}\tr(A)\omega+\theta(A^t)\omega.
    \end{align*}
  Thus, the matrix expression for $\tau_2$ is obtained by noticing that $\theta(A^t)\omega=-\omega(A^t\cdot,\cdot)-\omega(\cdot, A^t\cdot)=-\langle(JA^t+AJ)\cdot,\cdot\rangle$ and 
    $$
    (J^*\alpha(A)\wedge e^7)_{l7}=J^*\alpha(A)(e_l)=\alpha(A)(Je_l)=\langle \alpha^\sharp,Je_l\rangle=-(J\alpha^\sharp)_l \qforq l\in\{1,\dots,6\}.
    $$
    Lastly, we replace the expressions of $\tau_0$ and $\tau_1$ in order to derive the torsion $3$-form
    \begin{align*}
    	\tau_3    	      =&-\frac{2\tr(JA)}{7}(\omega\wedge e^7+\rho_+)+\frac{1}{4}\ast(\alpha(A)\wedge\omega\wedge e^7)+\frac{1}{4}\ast(\alpha(A)\wedge\rho_+)\\
           &+\frac{1}{2}\tr(A)\ast(e^7\wedge\rho_+)-\ast(\theta(A)\rho_+\wedge e^7)\\
    	      =&-\frac{2\tr(JA)}{7}(\omega\wedge e^7+\rho_+)-\frac{1}{4}\star(\alpha(A)\wedge\omega)+\frac{1}{4}\star(\alpha(A)\wedge\rho_+)\wedge e^7+\frac{1}{2}\tr(A)\rho_-+\star\theta(A)\rho_+.
    \end{align*}
 Now, applying the definition $\jmath(\tau_3)(e_a,e_b)=\ast(e_a\lrcorner\varphi\wedge e_b\lrcorner\varphi\wedge\tau_3)$, for $a=b=7$, we have
    \begin{align*}
    j(\tau_3)_{77}=-\frac{2\tr(JA)}{7}\ast(\omega^3\wedge e^7)=-\frac{12\tr(JA)}{7}.
    \end{align*}
    For $a=7$ and $b\neq 7$, using the identities from Lemmas \ref{Fino_Lemma}-\ref{Lauret_Lemmata}, we have
    \begin{align*}
    	j(\tau_3)_{7b}        =&-\frac{1}{4}\ast(\omega\wedge (e_b\lrcorner\omega)\wedge e^7\wedge\star(\alpha(A)\wedge\omega))+\ast(\omega\wedge (e_b\lrcorner\omega)\wedge e^7\wedge\star\theta(A)\rho_+)\\
        &+\frac{1}{4}\ast(\omega\wedge (e_b\lrcorner\rho_+)\wedge\star(\alpha(A)\wedge\rho_+)\wedge e^7)\\
    	=&\frac{1}{4}\star((e_b\lrcorner\omega)\wedge\star(\alpha(A)\wedge\omega)\wedge\omega)-\star((e_b\lrcorner\omega)\wedge\omega\wedge\star\theta(A)\rho_+)\\
    	&+\frac{1}{4}\star((e_b\lrcorner\rho_+)\wedge\star(\alpha(A)\wedge\rho_+)\wedge\omega)\\
    	=&-\frac{1}{2}Je_b\lrcorner\star^2\alpha(A)-Je_b\lrcorner\alpha-\frac{1}{4}\star((e_b\lrcorner\rho_+)\wedge J^\ast\alpha(A)\wedge\rho_-)\\
    	=&-\frac{1}{2}Je_b\lrcorner\alpha(A)+\frac{1}{4}(J^\ast\alpha(A))^\sharp\lrcorner\star(\star(e^b\wedge\rho_-)\wedge\rho_-)=-(J^\ast\alpha(A))_b=(J\alpha(A)^\sharp)_b.
    \end{align*}
     It remains the case $a\neq 7\neq b$, hence
    \begin{align*}
    	j(\tau_3)_{ab}    	=&\underbrace{-\tau_0\ast((e_a\lrcorner\omega)\wedge e^7\wedge(e_b\lrcorner\rho_+)\wedge\rho_+)}_{(\rI)}\underbrace{-\frac{1}{4}\ast((e_a\lrcorner\omega)\wedge e^7\wedge(e_b\lrcorner\rho_+)\wedge\star(\alpha(A)\wedge\omega))}_{(\rI\rI)}\\
    	              &\underbrace{+\frac{1}{2}\tr(A)\ast((e_a\lrcorner\omega)\wedge e^7\wedge(e_b\lrcorner\rho_+)\wedge\rho_-)}_{(\rI\rI\rI)}\underbrace{+\ast((e_a\lrcorner\omega)\wedge e^7\wedge(e_b\lrcorner\rho_+)\wedge\star\theta(A)\rho_+)}_{(\rI\rV)}\\
    	              &\underbrace{-\tau_0\ast((e_b\lrcorner\omega)\wedge e^7\wedge(e_a\lrcorner\rho_+)\wedge\rho_+)}_{(\rV)}\underbrace{-\frac{1}{4}\ast((e_b\lrcorner\omega)\wedge e^7\wedge(e_a\lrcorner\rho_+)\wedge\star(\alpha(A)\wedge\omega))}_{(\rV\rI)}\\
    	              &\underbrace{+\frac{1}{2}\tr(A)\ast((e_b\lrcorner\omega)\wedge e^7\wedge(e_a\lrcorner\rho_+)\wedge\rho_-)}_{(\rV\rI\rI)}\underbrace{+\ast((e_b\lrcorner\omega)\wedge e^7\wedge(e_a\lrcorner\rho_+)\wedge\star\theta(A)\rho_+)}_{(\rV\rI\rI\rI)}\\
    	              &\underbrace{-\tau_0\ast((e_a\lrcorner\rho_+)\wedge(e_b\lrcorner\rho_+)\wedge\omega\wedge e^7)}_{(\rI\rX)}\underbrace{+\frac{1}{4}\ast((e_a\lrcorner\rho_+)\wedge(e_b\lrcorner\rho_+)\wedge\star(\alpha(A)\wedge\rho_+)\wedge e^7)}_{(\rX)}
    \end{align*}
    For the first term, we have
    \begin{align*}
    	(\rI)=&\tau_0\star((e_a\lrcorner\omega)\wedge(e_b\lrcorner\rho_+)\wedge\rho_+)=-\tau_0(e_a\lrcorner\omega)^\sharp\lrcorner\star((e_b\lrcorner\rho_+)\wedge\rho_+)\\
    	   =&-\tau_0 Je_a\lrcorner\star(\star(e^b\wedge\rho_-)\wedge\rho_+)=-2\tau_0Je_a\lrcorner\star^2J^\ast e^b=-2\tau_0h_{ba},
    \end{align*}
    similarly $(\rV)=-2\tau_0h_{ab}$. For the second term, we have
    \begin{align*}
    	(\rI\rI)=&\frac{1}{4}\star((e_a\lrcorner\omega)\wedge(e_b\lrcorner\rho_+)\wedge\star(\alpha(A)\wedge\omega))=-\frac{1}{4}Je_a\lrcorner\star(\star(e^b\wedge\rho_-)\wedge(\alpha(A)^\sharp\lrcorner\omega)\wedge\omega)\\
    	    =&-\frac{1}{4}Je_a\lrcorner\star((\alpha(A)^\sharp\lrcorner\omega)\wedge J^\ast e^b\wedge\rho_+)=\frac{1}{4}Je_a\lrcorner(J^\ast\alpha(A))^\sharp\lrcorner\star(J^\ast e^b\wedge\rho_+)\\
    	    =&\frac{1}{4}Je_a\lrcorner(J^\ast\alpha(A))^\sharp\lrcorner Je_b\lrcorner\rho_-=\frac{1}{4}\rho_-(Je_b,(J^\ast\alpha(A))^\sharp,Je_a)\\
         =&-\frac{1}{4}J^\ast\rho_-(\alpha(A)^\sharp,e_a,e_b)=\frac{1}{4}(\alpha(A)^\sharp\lrcorner\rho_+)_{ab},
    \end{align*}
    similarly $(\rV\rI)=\frac{1}{4}(\alpha(A)^\sharp\lrcorner\rho_+)_{ba}$. For the third term, we have
    \begin{align*}
    	(\rI\rI\rI)=&-\frac{1}{2}\tr(A)\star((e_a\lrcorner\omega)\wedge\star(e^b\wedge\rho_-)\wedge\rho_-)\\
    	     =&\tr(A)(e_a\lrcorner\omega)^\sharp\lrcorner\star^2e^b=-\tr(A)\omega_{ab},
    \end{align*}
    similarly $(\rV\rI\rI)=-\tr(A)\omega_{ba}$. For the fourth term, we have
    \begin{align*}
        (\rI\rV)=&-\star((e_a\lrcorner\omega)\wedge(e_b\lrcorner\rho_+)\wedge\star\theta(A)\rho_+)\\
            =&-\langle (e_a\lrcorner\omega)\wedge(e_b\lrcorner\rho_+),\theta(A)\rho_+\rangle\\
            =&\frac{1}{2}\omega_{ai}\rho^+_{bjk}(A_{i}^l\rho^+_{ljk}+A_{j}^l\rho^+_{ilk}+A_{k}^l\rho^+_{ijl})\\
            =&\frac{1}{2}(4\omega_{ai}A_{i}^lh_{bl}+A_{j}^l\rho^-_{kla}\rho^+_{bjk}-A_{k}^l\rho^-_{jla}\rho^+_{bjk})\\
            =&\frac{1}{2}(4(AJ)_{ab}+A_{j}^l(-\omega_{lb}h_{aj}+\omega_{ab}h_{lj}+\omega_{lj}h_{ab}-\omega_{aj}h_{lb})\\
             &+A_{k}^l(-\omega_{lb}h_{ak}+\omega_{ab}h_{lk}+\omega_{lk}h_{ab}-\omega_{ak}h_{lb}))\\
            =&\frac{1}{2}(4(AJ)_{ab}-(JA)_{ab}+\tr(A)\omega_{ab}+\tr(JA)h_{ab}-(AJ)_{ab}\\
             &-(JA)_{ab}+\tr(A)\omega_{ab}+\tr(JA)h_{ab}-(AJ)_{ab}) \\
            =&\tr(JA)h_{ab}+\tr(A)\omega_{ab}-[J,A]_{ab} ,
    \end{align*}
    similarly, we obtain $(\rV\rI\rI\rI)=\tr(JA)h_{ba}+\tr(A)\omega_{ba}-[J,A]_{ba}$. For the ninth term, we have
    \begin{align*}
        (\rI\rX)=&-\tau_0\star(\star(e^a\wedge\rho_-)\wedge\star(e^b\wedge\rho_-)\wedge\omega)\\
           =&-\tau_0\star(\star(e^a\wedge\rho_-)\wedge J^\ast e^b\wedge\rho_+)=-2\tau_0Je_b\lrcorner\star^2J^\ast e^a=-2\tau_0h_{ab}.
     \end{align*}
     And, for the last term, we have
     \begin{align*}
         (\rX)=&\frac{1}{4}\star((e_a\lrcorner\rho_+)\wedge(e_b\lrcorner\rho_+)\wedge\star(\alpha(A)\wedge\rho_+))\\
             =&\frac{1}{4}\star((e_b\lrcorner(\rho_+\wedge\star(e^a\wedge\rho_-))\wedge\star(\alpha(A)\wedge\rho_+))-\frac{1}{4}\star((e_b\lrcorner e_a\lrcorner\rho_+)\wedge\rho_+\wedge\star(\alpha(A)\wedge\rho_+))\\
             =&\frac{1}{2}\star((e_b\lrcorner\star J^\ast e^a)\wedge\star(\alpha(A)\wedge\rho_+))-\frac{1}{2}\star((e_b\lrcorner e_a\lrcorner\rho_+)\wedge\star\alpha(A))\\
             =&-\frac{1}{2}\star(\star(e^b\wedge J^\ast e^a)\wedge\star(\alpha(A)\wedge\rho_+))+\frac{1}{2}(e_b\lrcorner e_a\lrcorner\rho_+)^\sharp\lrcorner\star^2\alpha(A)\\
              =&-\frac{1}{2}\star(e^b\wedge J^\ast e^a\wedge\alpha(A)\wedge\rho_+)-\frac{1}{2}(e_b\lrcorner e_a\lrcorner\rho_+)^\sharp\lrcorner\alpha(A)\\
              =&\frac{1}{2}(\alpha(A)^\sharp\lrcorner\rho_-)(Je_a,e_b)-\frac{1}{2}(\alpha(A)^\sharp\lrcorner\rho_+)(e_a,e_b)\\
              =&\langle J[J,C(A)]Je_a,e_b\rangle-\langle[J,C(A)]e_a,e_b\rangle=0
        \end{align*}
        In conclusion, for $a\neq 7\neq b$, we have
        $$
          j(\tau_3)_{ab}=-6\tau_0h_{ab}+2\tr(JA)h_{ab}-[J,A]_{ab}-[J,A]_{ba}
        $$
        note that $[J,A]^t=[J,A^t]$. So, joining the above computations, we obtain the expression $\frac{1}{4}j(\tau_3)$. 
\end{proof}

As a consequence of Proposition \ref{torsion_forms}, we can write the full torsion tensor
\begin{equation}\label{eq: full_torsion_tensor_def}
    T=\frac{\tau_0}{4}g_\varphi-\frac{1}{4}\jmath(\tau_{3})-(\tau_1)^\sharp\lrcorner\varphi-\frac{1}{2}\tau_2,
\end{equation}
in terms of the Lie bracket $A\in \gl(\R^6)$.

\begin{corollary}\label{full_torsion_tensorr}
	The full torsion tensor of $(\fg,A,\varphi)$ is
	\begin{equation}\label{eq: torsion_tensor_A}
	T= \frac{1}{2}\left(\begin{array}{c|c}
	[J,S(A)]+[J,C(A)]+(JA^t+AJ) & -J\alpha(A)^\sharp \\ \hline
	0 & \tr(JA)
	\end{array}
	\right)
	\end{equation}
\end{corollary}  

\begin{proof}
   We just need to compute
    \begin{align*}
        (\tau_1)^\sharp\lrcorner\varphi=&-\frac{1}{12}(\alpha(A)^\sharp\lrcorner\omega)\wedge e^7-\frac{1}{12}\alpha(A)^\sharp\lrcorner\rho_+-\frac{1}{6}\tr(A)\omega\\
        =&-\frac{1}{12}(J\alpha(A)^\sharp)^\flat\wedge e^7-\frac{1}{6}[J,C_A]-\frac{1}{6}\tr(A)J.
    \end{align*}
    Therefore, the matrix \eqref{eq: torsion_tensor_A} is obtained by replacing the expression of $\tau_0$, $\jmath(\tau_3)$ and $\tau_2$ (cf Proposition \ref{torsion_forms}) into \eqref{eq: full_torsion_tensor_def}.
\end{proof}

\section{Fernández-Gray classes of invariant $\rG_2$-structures on $(\fg,A)$}

 According to the decomposition of $\gl(\fg)$ into the irreducible $\gt$-invariant submodules 
\begin{equation*}
\gl(\fg)=\cW_1\oplus\cW_2\oplus \cW_3\oplus\cW_4=:\cW,
\end{equation*}
each torsion term of \eqref{eq: full_torsion_tensor_def} belongs to an irreducible component of $\cW$, according to the Fernández-Gray notation \cite{fernandez1982riemannian}*{$\S$ 4}, it is given by
\begin{equation*}
    \frac{\tau_0}{4}g_\varphi\in \cW_1, \quad -\frac{1}{2}\tau_2\in\cW_2, \quad -\frac14\jmath(\tau_3)\in\cW_3 \qandq -(\tau_1)^\sharp\lrcorner\varphi\in \cW_4.
\end{equation*}
This provides $16$-torsion classes of a $\gt$-structure. Alternatively, Fernández and Gray \cite{fernandez1982riemannian} described the $16$-torsion classes of a $\gt$-structure, defining relations of $\nabla\varphi$, $d\varphi$ or/and $d\psi$. 
On the other hand, with respect to the $\SU(3)$-invariant decomposition \eqref{eq: splitting.gl(6)}, there are $36$-classes of $7$-dimensional almost Abelian Lie algebras $(\fg,A)$. The next theorem classifies the $16$-torsion classes of an invariant $\gt$-structure in terms of the $\SU(3)$-invariant decomposition of the bracket $A$: 

\begin{theorem}\label{thm: 16-torsion classes}
Let $(\fg,A,\varphi)$ be an almost Abelian Lie algebra with $\gt$-structure. Thus, for each torsion form we have: 
\begin{equation}\label{eq: torsion forms vanish conditions}
    \begin{array}{ccc}
            \tau_0=0 & \Leftrightarrow & \tr(JA)=0 \\ 
             \tau_1=0 & \Leftrightarrow & \tr(A)=0, \quad C_-(A)=0 \\
             \tau_2=0 & \Leftrightarrow & C_-(A)=S_+(A)=0 \\ 
             \jmath(\tau_3)=0 & \Leftrightarrow & \tr(JA)=0, \quad C_-(A)=S_-(A)=0.
        \end{array}
\end{equation}
In particular, within the $16$-torsion classes, those classes on Table \ref{tb:torsion_classes_introduction} are realised.
\end{theorem}

\begin{proof}
Denote by $A=\frac{\tr(A)}{6}I+S_+(A)+S_-(A)+\frac{\tr(JA)}{6}J+C_+(A)+C_-(A)$ the $\SU(3)$-invariant decomposition for the Lie bracket of $\fg$. Thus,
the torsion forms $\tau_2$ and $\tau_3$ from Proposition \ref{torsion_forms} can be rewritten as:
 \begin{align}\label{eq: tau2 and tau3}
 \begin{split}
	\tau_2=& -\frac{1}{3}\left(\begin{array}{c|c}
	-2JC_-(A)+6JS_+(A) & -J\alpha^\sharp \\ \hline
	\Big(J\alpha^\sharp\Big)^t & 0
	\end{array}
	\right),\\
	\frac{1}{4}\jmath(\tau_{3})=& \left(\begin{array}{c|c}
	\frac{1}{14}\tr(JA)I_6-JS_-(A) & \frac{1}{4}J\alpha^\sharp \\ \hline
	\frac{1}{4}\Big(J\alpha^\sharp\Big)^t & -\frac{3}{7}\tr(JA)
	\end{array}
	\right).
	\end{split}
    \end{align}
     According with Corollary \ref{cor: vanishing alpha}, Proposition \ref{torsion_forms} and equation \eqref{eq: tau2 and tau3}, we obtain the corresponding condition for $\tau_k=0$ (with $k=0,1,2,3$), given in \eqref{eq: torsion forms vanish conditions}.
\end{proof}

\begin{remark}
   \begin{enumerate}
    \item Notice that the torsion classes $\cW_2$ (closed $\gt$-structures) and $\cW_1\oplus\cW_3$ (coclosed $\gt$-structures) correspond with 
    $$\sym_+^0\oplus \su(3)\simeq\fsl(3,\C) \qandq \sym^0_-\oplus\R\cdot J\oplus\su(3)\simeq\fsp(6,\R),$$ respectively. These results have been obtained independently by Freibert \cites{Freibert2012,Freibert2013}, using a different approach.
    \item The Lie algebra $(\fg,A,\varphi)$ with bracket in $\R\cdot I_6\oplus\so(6)$ has a $\gt$-structure with  torsion forms $\tau_1,\tau_2,\tau_3$. In particular, $A$ is normal and using expression for the Ricci curvature \cite{Arroyo2013}*{Eq (8)}:
    \begin{equation}\label{Ricci_operator}
	\ricci_A=\left(\begin{array}{c|c}
	\frac{1}{2}\bigg([A,A^t]-\tr(A)(A+A^t)\bigg) & 0 \\ \hline
	0 & \color{blue}-\frac{1}{4}\tr\left((A+A^t)^2\right)
	\end{array}\right),
	\end{equation}
	we see that the induced $\gt$-metric is  Einstein, with constant $c=-(\tr A)^2$ (e.g. \cite{delBarco2023}*{Lemma 6.1}).
	\item If $B\in \so(6)$, then the induced $\gt$-metric is Ricci flat, in particular, $(\fg,B,g_\varphi)$ is flat. However, $\varphi$ is not torsion-free, since the terms $\tr(JB)$ and $C_-(B)$ could be non-zero.
    \end{enumerate}
\end{remark}

\begin{corollary}\label{cor: No-existence_torsion_classes}
    There does not exist an invariant $\gt$-structure on an almost Abelian Lie group with torsion strictly in one of the following classes:
    \begin{equation*}
        \cW_1, \quad \cW_1\oplus \cW_4, \quad \cW_1\oplus\cW_2 \qandq \cW_1\oplus\cW_2\oplus\cW_4.
    \end{equation*}
\end{corollary}

\begin{remark} Independent of the context of almost Abelian Lie groups, some of the torsion classes in Corollary \ref{cor: No-existence_torsion_classes} never happen in manifolds with some prescribed geometry and/or topology:
\begin{enumerate}
    \item 
    A $\gt$-structure with torsion in $\cW_1$ has non-negative the scalar curvature  $\Scal(g)=\frac{28}{9}\tau_0^2$. Unlike, every left invariant metric on a solvable Lie group is either flat or else has strictly negative scalar curvature \cite{milnor1976curvatures}*{Theorem 3}.
    \item 
    An invariant $\gt$-structure $\varphi$ with torsion in  $\cW_1\oplus \cW_4$ has torsion either $\cW_1$ or $\cW_4$. Indeed, taking the exterior derivative of \eqref{eq: torsion forms defi} with $\tau_2=0$ and $\tau_3=0$, we obtain the equations
    $$
    d\tau_1=0 \qandq d\tau_0=\tau_0\tau_1.
    $$
    Since $\varphi$ is invariant then $\tau_0$ is constant and thus $\tau_1=0$ or $\tau_0=0$.
    \item 
   The $\gt$-structures with torsion class in $\cW_1\oplus \cW_2$ on connected manifolds reduce to either $\cW_1$ or $\cW_2$ \cite{Cabrera1996}*{Theorem 2.1}.
\end{enumerate}
\end{remark}

\begin{definition}
    A \emph{lattice} $\Gamma$ of a Lie group $G$ is a discrete subgroup $\Gamma\subset G$, such that the quotient $\Gamma\backslash G$ is compact.
\end{definition}

There exist obstructions for the existence of a lattice $\Gamma\subset G$, in particular, to achieve it, $G$ needs to be a \emph{unimodular} Lie group \cite{milnor1976curvatures}*{Lemma 6.2}, (i.e., $\det(\Ad(g))=\pm 1$ for every $g\in G$). Equivalently, when $G$ is connected, it is unimodular if, and only if, the linear transformation is $\tr(\ad(u))=0$ for every $u\in \fg$. In this case, $\fg$ is called a unimodular Lie algebra. Using Theorem \ref{thm: 16-torsion classes}, we can determinate which torsion classes appear for a $\gt$-structure on a unimodular almost Abelian Lie algebra:
\begin{corollary}\label{cor: unimodular-torsion classes}
Let $(\fg,A,\varphi)$ be a unimodular almost Abelian Lie algebra with $\gt$-structure, then $\varphi$ belongs to one of the following torsion classes
\begin{table}[h!]
\begin{tabular}{||c|c|c||}
\hline 
\text{Class} &  \text{Vanishing torsion} \quad & \text{Bracket relation}    \\ [0.5ex] \hline\hline
    $\cW=\{0\}$ & $\tau_0=0,\tau_1=0,\tau_2=0,\tau_3=0$  &  $A\in\su(3)$ \\ [0.5ex]
    $\cW_2$ & $\tau_0=0,\tau_1=0,\tau_3=0$ & $A\in \sym_+^0\oplus \su(3)$\\ [0.5ex]
    $\cW_3$ & $\tau_0=0,\tau_1=0,\tau_2=0$ & $A\in \sym_-^0\oplus\su(3)$\\ [0.5ex]
    $\cW_1\oplus\cW_3$ & $\tau_1=0,\tau_2=0$ & $A\in \sym_-^0\oplus \R\cdot J\oplus \su(3)$\\ [0.5ex]
    $\cW_2\oplus\cW_3$ & $\tau_0=0,\tau_1=0$ & $A\in \sym_+^0\oplus \sym_-^0\oplus \su(3)$ \\ [0.5ex]
    $\cW_1\oplus\cW_2\oplus\cW_3$ & $\tau_1=0$ &  $A\in \sym_+^0\oplus \sym_-^0\oplus \R\cdot J\oplus\su(3)$ \\ [0.5ex]
    $\cW_2\oplus\cW_3\oplus\cW_4$ & $\tau_0=0$ & $A\in \sym^0_+\oplus \sym^0_-\oplus\fm\oplus \su(3)$ \\ [0.5ex]
    $\cW_1\oplus\cW_2\oplus\cW_3\oplus\cW_4$ & \text{No vanishing condition} & $A\in  \sym^0_+\oplus \sym^0_-\oplus \R\cdot J\oplus\fm\oplus \su(3)$ \\ [1ex] \hline
\end{tabular}
\caption{Torsion classes of $(\fg,A,\varphi)$ unimodular}
\end{table}
\end{corollary}

\begin{example}\label{ex: unimodular_example}
  Consider the almost Abelian Lie algebra $(\fg,A)$ with the $\gt$-structure \eqref{coclosed_almost_abelian_G2} and 
  $$A=\left(\begin{array}{c|cccc}
      0_{2\times 2} & & & &  \\
     \hline
        & 0 & 0 & -1 & 0\\
        & 0 & 0 & 0 & 1\\
        & 1 & 0 & 0 & 0\\
        & 0 & -1 & 0 & 0
  \end{array}\right).
  $$
  By Proposition \ref{torsion_forms}, we have 
  \begin{align*}
      \tau_0=0, \quad \tau_1=\frac13 e^2, \quad \tau_2=\frac{2}{3}\left(e^{36}+e^{45}-2e^{17}\right) \qandq 
      \jmath(\tau_3)=4(e^1\otimes e^7+e^7\otimes e^1).
  \end{align*}
  Moreover, $\ricci_A=0$ and the associated connected and simply connected Lie group $G_A=\R^6\times_\mu \R$, where $\mu:\R\rightarrow \R^6$ satisfies $d(\mu(t))_0=e^{tA}$. It admits a discrete subgroup $\Gamma_a =\exp(\Z^6\times_\mu a\Z)$ ( for $a\in\{\frac{\pi}{2},\pi,\frac{3\pi}{2},2\pi\}$), such that $(\Gamma_a\backslash G_A, \varphi)$ is a flat, compact solvmanifold with invariant $\gt$-structure with torsion in $\cW_2\oplus\cW_3\oplus\cW_4$.
\end{example}

\section{Divergence of the full torsion tensor}\label{divergence_Laplacian}

In this section, we compute the divergence of the full torsion tensor of $\varphi$ in terms of Lie algebra data $(\fg,A)$. It yields a characterisation of the invariant harmonic $\gt$-structures on almost Abelian Lie groups. Recall that the  Levi-Civita connection is given by the left-invariant metric of $(\fg,A)$ \cite{milnor1976curvatures}*{Lemma 5.5}, which is given by
	\begin{equation}\label{eq: LC_connection}
	    \nabla_7e_7=0, \quad \nabla_ie_7=-S(A)(e_i), \quad \nabla_7e_i=C(A)(e_i) \qandq \nabla_ie_j=\langle S(A)(e_i),e_j\rangle e_7
	\end{equation}
	where $i,j=1,...,6$. Thus, we have:

\begin{theorem}\label{divergence_vector}
    Let $(\fg,A,\varphi)$ be an almost Abelian Lie algebra with $\gt$-structure, then the divergence of the full torsion tensor \eqref{eq: torsion_tensor_A} is 
    \begin{equation}\label{eq: divergence_of_T_1-form}
        \diver T=-\frac{1}{2}\tr(A)J^\ast\alpha(A)+\frac{1}{2}\theta(C(A))J^\ast\alpha(A)-\frac{1}{2}\tr(A)\tr(JA)e^7.
    \end{equation}
    Furthermore, $\varphi$ is harmonic if, and only if, $A\in \gl(\R^6)$ satisfies 
    \begin{equation}\label{eq: divergence condition}
        \tr(JA)\tr(A)=0 \qandq -\tr(A)J^\ast\alpha(A)+\theta(C(A))(J^\ast\alpha(A))=0,
    \end{equation}
    with $\alpha(A)$ defined in \eqref{alpha_form}.
\end{theorem}

\begin{proof}
By definition $\diver T=\sum_{j=1}^7(\diver T)_je^j=\sum_{i,j=1}^7\nabla_iT_{ij}e^j$, thus using \eqref{eq: torsion_tensor_A} and  \eqref{eq: LC_connection}, we have
	\begin{align*}
		\diver T_7=&\sum_{i=1}^7\nabla_iT_{i7}=- \sum_{i=1}^7T(\nabla_ie_i,e_7)+T(e_i,\nabla_ie_7)\\
		          =&-\sum_{i=1}^6\langle S(A)(e_i),e_i\rangle T_{77}-T(e_i,S(A)(e_i))\\		          =&-\frac{1}{2}\tr(S(A))\tr(JA)+\frac{1}{2}\langle [J,S(A)]+[J,C(A)]+JA^t+AJ,S(A)\rangle\\
		          =&-\frac{1}{2}\tr(S(A))\tr(JA),
	\end{align*}
    since $\langle [J,S(A)],S(A)\rangle=-\langle J,[S(A),S(A)]\rangle=0$ and  $[J,C(A)]+JA^t+AJ$ is skew-symmetric. Now,	for $k\neq 7$, we have
	\begin{align*}
		\diver T_k=&\sum_{i=1}^7\nabla_ iT_{ik}=-\sum_{i=1}^7T(\nabla_ie_i,e_k)+T(e_i,\nabla_ie_k)\\
		          =&-T(e_7,C(A)(e_k))-\sum_{i=1}^6\langle S(A)(e_i)e_i\rangle T_{7k}+\langle S(A)(e_i),e_k\rangle T_{i7}\\
		          =&\frac{1}{2}\langle J\alpha(A)^\sharp,C(A)(e_k)\rangle+\frac{1}{2}\tr(S(A))\langle J\alpha(A)^\sharp,e_k\rangle\\
		      =&-\frac{1}{2}\alpha(A)(JC(A)(e_k))-\frac{1}{2}\tr(S(A))\alpha(A)(Je_k)
	\end{align*}
\end{proof}

\begin{remark}
The $1$-form \eqref{eq: divergence_of_T_1-form} can be rewritten as a vector
$$
\diver T^\sharp=\frac{1}{2}\tr(A)J(\alpha(A)^\sharp)-\frac{1}{2}C(A)J(\alpha(A)^\sharp)-\frac{1}{2}\tr(A)\tr(JA)e_7.
$$
Therefore, the divergence free condition \eqref{eq: divergence condition} can be rewritten as:
\begin{equation}\label{eq: vector_divergence condition}
  \tr(A)\tr(JA)=0 \qandq (\tr(A)J-C(A)J)(\alpha(A)^\sharp)=0.
  \end{equation}
In particular, since $\langle C(A)J(\alpha(A)^\sharp),J(\alpha(A)^\sharp)\rangle=0$, the harmonic condition \eqref{eq: vector_divergence condition} can be rewritten as
\begin{equation}
   \tr(A)\tr(JA)=0, \quad \tr(A)\alpha(A)=0 \qandq C(A)J\alpha(A)^\sharp=0
\end{equation}
\end{remark}
\begin{corollary}
Let $(\fg,A,\varphi)$ be a unimodular almost Abelian Lie algebra with $\gt$-structure, thus $\varphi$ is harmonic if, and only if, $A$ satisfies 
    \begin{equation}\label{eq: unimodular_divergence condition}
    C(A)J\alpha(A)^\sharp=0,
    \end{equation}
    equivalently, $(C(A))_{ln}\rho^-_{nij}[J,C(A)]_{ij}=0$ for each $ l\in\{1,...,6\}$.
\end{corollary}

Using equation \eqref{eq: divergence condition}, we can see which torsion classes from table \ref{tb:torsion_classes_introduction}  are harmonic. 

\begin{theorem}\label{thm:harmonic_torsion_classes}
 Let $(\fg,A,\varphi)$ be an almost Abelian Lie algebra with $\gt$-structure, thereby, $\varphi$ is harmonic if it belongs to one of the following torsion classes:
\begin{eqnarray*}
     \{0\}, \quad \cW_2, \quad \cW_3, \quad \cW_4,\\ \cW_1\oplus\cW_3, \quad \cW_2\oplus\cW_4, \quad \cW_3\oplus\cW_4, \\ \cW_2\oplus\cW_3, \quad \cW_1\oplus\cW_2\oplus\cW_3. 
\end{eqnarray*}
 Further, if $\varphi$ is of type $\cW_1\oplus\cW_3\oplus\cW_4$ and $\diver T=0$, then $\varphi$ is of type $\cW_1\oplus\cW_3$ or $\cW_3\oplus\cW_4$.
\end{theorem}

Next, we illustrate   examples of $(\fg,A,\varphi)$ with torsion in the class $\cW_2\oplus\cW_3\oplus\cW_4$ with $\diver T=0$ and non-vanishing divergence torsion tensor:

\begin{example}[$\gt$-structure with torsion in $\cW_2\oplus\cW_3\oplus\cW_4$] Consider the two following examples:
\begin{enumerate}
    \item[(i)] From example \ref{ex: unimodular_example}, we have $\alpha(A)^\sharp=4e_2$ and $\tr(A)=0$, since $J\alpha(A)^\sharp\in \ker A$, by \eqref{eq: vector_divergence condition}, we have $\diver T=0$.
    \item[(ii)]  Consider the almost Abelian Lie algebra $(\fg,B)$ with the $\gt$-structure \eqref{coclosed_almost_abelian_G2} and 
  $$B=\left(\begin{array}{cc|cccc}
     1 & 0 & & & &  \\
     0 & 1 & & & & \\
     \hline
    &    & 0 & 0 & -1 & 0\\
     &   & 0 & 0 & 0 & 1\\
      &  & 1 & 0 & 0 & 0\\
       & & 0 & -1 & 0 & 0
  \end{array}\right).
  $$ Thus, $\alpha(B)^\sharp=4e_2$ and, using equation \eqref{eq: vector_divergence condition}, we see that  $\diver T^\sharp=\frac{\tr(B)}{2}J(\alpha(B)^\sharp)=-4e_1$, since $\tr(B)=2$ and $J\alpha^\sharp\in\ker( C(B))$.
\end{enumerate}
\end{example}

\begin{example}[$\gt$-structure with torsion in $\cW$]
Consider the almost Abelian Lie algebra $(\fg,D)$ with the $\gt$-structure \eqref{coclosed_almost_abelian_G2} and 
$$D=
\left(\begin{array}{cc|cccc}
     0 & 0 & & & &  \\
     0 & 0 & & & & \\
     \hline
    &    & 0 & -1 & -1 & 0\\
     &   & 1 & 0 & 0 & 1\\
      &  & 1 & 0 & 0 & -1\\
       & & 0 & -1 & 1 & 0
  \end{array}\right).
  $$ Since $J(\alpha(D)^\sharp)=-4e_1\in \ker(C(D))$ and $\tr(D)=0$, by \eqref{eq: vector_divergence condition} we have $\diver T^\sharp=0$.
\end{example}

\begin{remark}
 We recall that, the torsion classes $\{0\}$, $\cW_2$, $\cW_3$, $\cW_4$ and $\cW_2\oplus \cW_3$ are generically harmonic by \cite{Grigorian2019}*{Theorem 4.3}. Besides, since $\tau_0$ is constant for Lie groups, then the torsion classes $\cW_1\oplus\cW_3$ and $\cW_1\oplus\cW_2\oplus\cW_3$ also satisfy $\diver T=0$ (see \eqref{eq: Grigorian_divergence_free_cases}).  Therefore, the almost Abelian Lie algebras $(\fg,A,\varphi)$ whose $\gt$-structure has torsion in the classes  $\cW_2\oplus\cW_4$, and $\cW_3\oplus\cW_4$ are new examples of harmonic $\gt$-structures. However, these new examples $(\fg,A,\varphi)$ do not admit a lattice, since they are not unimodular (see Theorem \ref{thm: 16-torsion classes}). 
\end{remark}

\newpage
\bibliography{biblio.bib}




\end{document}